\documentclass{amsart}
\usepackage{graphics}
\usepackage{amsfonts,amsmath}
\begin{document}

 \newtheorem{thm}{Theorem}[section]
 \newtheorem{cor}[thm]{Corollary}
 \newtheorem{lem}[thm]{Lemma}{\rm}
 \newtheorem{prop}[thm]{Proposition}

 \newtheorem{defn}[thm]{Definition}{\rm}
 \newtheorem{assumption}[thm]{Assumption}
 \newtheorem{rem}[thm]{Remark}
 \newtheorem{ex}{Example}
\numberwithin{equation}{section}

\def\la{\langle}
\def\ra{\rangle}
\def\e{{\rm e}}
\def\x{\mathbf{x}}
\def\by{\mathbf{y}}
\def\bz{\mathbf{z}}
\def\F{\mathcal{F}}
\def\R{\mathbb{R}}
\def\T{\mathbf{T}}
\def\N{\mathbb{N}}
\def\K{\mathbf{K}}
\def\bK{\overline{\mathbf{K}}}
\def\Q{\mathbf{Q}}
\def\M{\mathbf{M}}
\def\O{\mathbf{O}}
\def\C{\mathbb{C}}
\def\P{\mathbf{P}}
\def\Z{\mathbb{Z}}
\def\H{\mathcal{H}}
\def\A{\mathbf{A}}
\def\V{\mathbf{V}}
\def\AA{\overline{\mathbf{A}}}
\def\B{\mathbf{B}}
\def\c{\mathbf{C}}
\def\L{\mathbf{L}}
\def\bS{\mathbf{S}}
\def\H{\mathcal{H}}
\def\I{\mathbf{I}}
\def\Y{\mathbf{Y}}
\def\X{\mathbf{X}}
\def\G{\mathbf{G}}
\def\f{\mathbf{f}}
\def\z{\mathbf{z}}
\def\y{\mathbf{y}}
\def\d{\hat{d}}
\def\bx{\mathbf{x}}
\def\y{\mathbf{y}}
\def\v{\mathbf{v}}
\def\g{\mathbf{g}}
\def\w{\mathbf{w}}
\def\b{\mathcal{B}}
\def\a{\mathbf{a}}
\def\q{\mathbf{q}}
\def\u{\mathbf{u}}
\def\s{\mathcal{S}}
\def\cc{\mathcal{C}}
\def\co{{\rm co}\,}
\def\cp{{\rm CP}}
\def\tg{\tilde{f}}
\def\tx{\tilde{\x}}
\def\supmu{{\rm supp}\,\mu}
\def\supnu{{\rm supp}\,\nu}
\def\m{\mathcal{M}}
\def\C{\mathcal{C}}
\def\s{\mathcal{S}}
\def\k{\mathcal{K}}
\def\la{\langle}
\def\ra{\rangle}

\title[nonnegativity]{New approximations for the cone of copositive matrices and its dual}
\author{Jean B. Lasserre}
\address{LAAS-CNRS and Institute of Mathematics\\
University of Toulouse\\
LAAS, 7 avenue du Colonel Roche\\
31077 Toulouse C\'edex 4,France}
\email{lasserre@laas.fr}
\date{}

\begin{abstract}
We provide convergent hierarchies for the convex cone $\cc$ of 
copositive matrices and its dual $\cc^*$, the cone of completely positive matrices. 
In both cases the corresponding hierarchy consists of nested spectrahedra
and provide outer (resp. inner) approximations for $\cc$ (resp. for its dual $\cc^*$), thus complementing
previous inner (resp. outer) approximations for $\cc$ (for $\cc^*$). In particular, both inner and outer approximations
have a very simple interpretation. Finally, extension to $\k$-copositivity and $\k$-complete positivity
for a closed convex cone $\k$, is straightforward.
\end{abstract}

\keywords{copositive matrices; completely positive matrices; semidefinite relaxations}

\subjclass{15B48 90C22}

\maketitle

\section{Introduction}

In recent years the convex cone $\cc$ of {\it copositive} matrices
and its dual cone $\cc^*$ of {\it completely positive} matrices 
have attracted a lot of attention,
in part because several interesting NP-hard problems can be modelled as convex conic optimization
problems over those cones. For a survey of results and discussion on $\cc$ and its dual, the interested reader is referred to e.g. Anstreicher and Burer \cite{anstreicher}, Bomze \cite{bomze1}, Bomze et al. \cite{bomze2},
Burer \cite{burer}, D\"ur \cite{durr} and Hiriart-Urruty and Seeger \cite{jbhu-sirev}..

As optimizing over $\cc$ (or its dual) is in general difficult, a
typical approach is to optimize over simpler and more tractable cones. In particular,
nested hierarchies of tractable convex cones $\mathcal{C}_k$, $k\in\N$, that provide {\it inner} approximations of $\cc$ have been proposed, notably by Parrilo \cite{parrilo}, deKlerk and Pasechnik \cite{de-pasech}, Bomze and deKlerk \cite{bomze}, as well
as Pe\~na et al. \cite{pena}. For example, denoting by 
$\mathcal{N}$ (resp. $\mathcal{S}_+$) the convex cone of nonnegative (resp. positive semidefinite) matrices,
the first cone in the hierarchy of \cite{de-pasech} is $\mathcal{N}$, and $\mathcal{N}+\mathcal{S}_+$ in that of \cite{parrilo},
whereas the hierarchy of \cite{pena} is in sandwich between that of \cite{de-pasech} and \cite{parrilo}.
Of course, to each such hierarchy of inner approximations $(\mathcal{C}_k)$, $k\in\N$, 
of $\cc$, one may associate the hierarchy $(\mathcal{C}^*_k)$, $k\in\N$, of dual cones 
which provides outer approximations of $\cc^*$. However, quoting D\"ur in \cite{durr}: ``We are not aware of comparable approximation schemes that approximate
the completely positive cone (i.e. $\cc^*$) from the interior."

Let us also mention Gaddum's characterization
of copositive matrices described in Gaddum \cite{gaddum} for which the set membership problem ``$\A\in\cc$" (with $\A$ known)
reduces to solving a linear programming (LP) problem (whose size is not polynomially bounded in the input size of the matrix $\A$).
However, Gaddum's characterization is {\it not} convex in the coefficients of the matrix $\A$, and so
cannot be used to provide a hierarchy of outer approximations of $\cc$. On the other hand,
de Klerk and Pasechnik \cite{de-pasek} have used Gaddum's characterization to help solve 
quadratic optimization problems on the simplex
via solving a hierarchy of LP-relaxations of increasing size and with {\it finite} convergence of the process.

{\bf Contribution.} The contribution of this note is precisely to describe an explicit hierarchy of tractable convex cones that provide
{\it outer} approximations of $\cc$ which converge monotonically and asymptotically to $\cc$. And so, by duality, the corresponding hierarchy of dual cones
provides {\it inner} approximations of $\cc^*$ converging to $\cc^*$, answering D\"ur's question and also showing  that
$\cc$ and $\cc^*$ can be sandwiched between two converging hierarchies of tractable convex cones. The present outer approximations are {\it not} polyhedral and hence are different from the outer polyhedral approximations 
$\mathcal{O}^n_r$ of $\cc$ defined in Yildirim \cite{yildirim}. $\mathcal{O}^n_r$ are based on a certain discretization $\Delta(n,r)$ of the simplex $\Delta_n:=\{\x\in\R^n_+:\,\e^T \x=1\}$ whose size is parametrized by $r$. All matrices associated with quadratic forms nonnegative on $\Delta(n,r)$ belong to $\mathcal{O}^r_n$ for all $r$, and $\cc=\cap_{r=1}^\infty\mathcal{O}^r_n$;
Combining with the inner approximations of de Klerk and Pasechnik
\cite{de-pasech}, the author provides tight bounds on the gap between upper and lower bounds
for quadratic optimization problems;  for more details the interested reader is
referred to Yildirim \cite{yildirim}.

In fact, our result is a specialization of a more general result of \cite{lasserre} about nonnegativity of polynomials 
on a closed set $\K\subset\R^n$, in the specific context of quadratic forms and $\K=\R^n_+$.
In such a context, and identifying copositive matrices with 
quadratic forms nonnegative on $\K=\R^n_+$, this specialization yields 
inner approximations for $\cc^*$  with a very simple interpretation directly related to the definition of $\cc^*$, 
which might be of interest for the community interested in $\cc$ and $\cc^*$ but
might be ``lost" in the general case treated in \cite{lasserre}, whence the present note.
Finally, following Burer \cite{burer},  and $\k\subset\R^n$ being a closed convex cone,
one may also consider the convex cone $\cc_\k$
of $\k$-copositive matrices, i.e., real symmetric matrices $\A$ such that
$\x^T\A\x\geq0$ on $\k$, and its dual cone $\cc^*_\k$ of $\k$-completely positive matrices.
Then the outer approximations previously defined have an immediate and straightforward analogue
(as well as the inner approximations of the dual).
%Therefore, in view of the recent new interest in $\cop$ and its dual, we think that it deserves a special treatment.
%Moreover, each

\section{Main result}
\subsection{Notation and definitions}
Let $\R[\x]$ be the ring of polynomials in the variables
$\x=(x_1,\ldots,x_n)$.
Denote by $\R[\x]_d\subset\R[\x]$ the vector space of
polynomials of degree at most $d$, which forms a vector space of dimension $s(d)={n+d\choose d}$, with e.g.,
the usual canonical basis $(\x^\alpha)$ of monomials.
Also, let $\N^n_d:=\{\alpha\in\N^n\,:\,\sum_i\alpha_i\leq d\}$ and
denote by $\Sigma[\x]\subset\R[\x]$ (resp. $\Sigma[\x]_d\subset\R[\x]_{2d}$)
the space of sums of squares (s.o.s.) polynomials (resp. s.o.s. polynomials of degree at most $2d$). 
If $f\in\R[\x]_d$, write
$f(\x)=\sum_{\alpha\in\N^n_d}f_\alpha \x^\alpha$ in the canonical basis and
denote by $\f=(f_\alpha)\in\R^{s(d)}$ its vector of coefficients. Finally, let $\s^n$ denote the space of 
$n\times n$ real symmetric matrices, with inner product $\la \A,\B\ra ={\rm trace}\,\A\B$, and where the notation
$\A\succeq0$ (resp. $\A\succ0$) stands for $\A$ is positive semidefinite.

Given $\K\subseteq\R^n$, denote by ${\rm cl}\,\K$ (resp. ${\rm conv}\,\K$) the closure (resp. the convex hull) of $\K$.
Recall that given a convex cone $\K\subseteq\R^n$, the convex cone
$\K^*=\{\y\in\R^n: \la\y,\x\ra\geq0\:\forall\x\in\K\}$ is called the {\it dual} cone of $\K$, and satisfies
$(\K^*)^*={\rm cl}\,\K$ . Moreover, given
two convex cones $\K_1,\K_2\subseteq\R^n$,
\[\begin{array}{ccccc}
\K_1^*\,\cap\,\K_2^*&=&(\K_1+\K_2)^*&=&(\K_1\cup\K_2)^*\\
(\K_1\cap\K_2)^*&=&{\rm cl}\,(\K_1^*+\K_2^*)&=&{\rm cl}\,({\rm conv}\,(\K_1^*\cup\K_2^*)\,).\end{array}\]
See for instance Rockafellar \cite[Theorem 3.8; Corollary 16.4.2]{rockafellar}.

\subsection*{Moment matrix} With a sequence $\y=(y_\alpha)$, $\alpha\in\N^n$,
let $L_\y:\R[\x]\to\R$ be the linear functional
\[h\quad (=\sum_{\alpha}h_{\alpha}\,\x^\alpha)\quad\mapsto\quad
L_\y(h)\,=\,\sum_{\alpha}h_{\alpha}\,y_{\alpha},\quad h\in\R[\x].\]
With $d\in\N$, let $\M_d(\y)$ be the symmetric matrix with rows and columns indexed 
in $\N^n_d$, and defined by:
\begin{equation}
\label{moment}\M_d(\y)(\alpha,\beta)\,:=\,L_\y(\x^{\alpha+\beta})\,=\,y_{\alpha+\beta},\quad (\alpha,\beta)\in\N^n_d\times \N^n_d.\end{equation}

The matrix $\M_d(\y)$ is called the moment matrix associated with $\y$, and it is straightforward to check that
\[\left[\,L_\y(g^2)\geq0\quad\forall g\in\R[\x]\,\right]\quad\Leftrightarrow\quad\M_d(\y)\,\succeq\,0,\quad d=0,1,\ldots.\]

\subsection*{Localizing matrix} Similarly, with $\y=(y_{\alpha})$, $\alpha\in\N^n$,
and $f\in\R[\x]$ written
\[\x\mapsto f(\x)\,=\,\sum_{\gamma\in\N^n}f_{\gamma}\,\x^\gamma,\]
let $\M_d(f\,\y)$ be the symmetric matrix with rows and columns indexed in $\N^n_d$,
and defined by:
\begin{equation}
\label{localizing}
\M_d(f\,\y)(\alpha,\beta)\,:=\,L_\y\left(f(\x)\,\x^{\alpha+\beta}\right)\,=\,\sum_{\gamma}f_{\gamma}\,
y_{\alpha+\beta+\gamma},\qquad (\alpha,\beta)\in\N^n_d\times \N^n_d.\end{equation}
The matrix $\M_d(f\,\y)$ is called the localizing matrix associated with $\y$ and $f\in\R[\x]$.
%, and it is straightforward to check that\[\left[\,L_\y(g^2\,f)\geq0\quad\forall g\in\R[\x]\,\right]\quad\Leftrightarrow\quad\M_d(f\,\y)\,\succeq\,0,\quad d=0,1,\ldots,\]
Observe that
\begin{equation}
\label{local1}
\la \g,\M_d(f\,\y)\g\ra\,=\,L_\y(g^2\,f),\qquad \forall g\in\R[\x]_d,
\end{equation}
and so if $\y$ has a representing finite Borel measure $\mu
$, i.e., if
\[y_\alpha\,=\,\int_{\R^n} \x^\alpha\,d\mu,\qquad \forall\,\alpha\in\N^n,\]
then (\ref{local1}) reads
\begin{equation}
\label{local2}
\la \g,\M_d(f\,\y)\g\ra\,=\,L_\y(g^2\,f)\,=\,\int_{\R^n} g(\x)^2f(\x)\,d\mu(\x),\qquad \forall g\in\R[\x]_d.
\end{equation}
Actually, the localizing matrix $\M_d(f\,\y)$ is nothing less than the moment matrix associated with the 
sequence $\z=f\,\y=(z_\alpha)$, $\alpha\in\N^n$, with $z_\alpha=\sum_\gamma f_\gamma y_{\alpha+\gamma}$.
In particular, if $f$ is nonnegative on the support $\supmu$ of $\mu$ then
the localizing matrix $\M_d(f\,\y)$ is just the moment matrix associated
with the finite Borel measure $d\mu_f=fd\mu$, absolutely continuous with respect to $\mu$ (denoted $\mu_f\ll\mu$), and with density $f$.
For instance, with $d=1$ one may write
\[\M_1(f\,\y)\,=\,
\displaystyle\int_{\R^n}\:\left[\begin{array}{ccc}1&\vert &\x^T\\ -& &-\\ \x &\vert &\x\x^T\end{array}\right]\,f(\x)\,d\mu(\x)\,=\,
\displaystyle\int_{\R^n}\:\left[\begin{array}{ccc}1&\vert &\x^T\\ -& &-\\ \x &\vert &\x\x^T\end{array}\right]\,d\mu_f(\x),\]
or, equivalently,
\begin{equation}
\label{second}
\M_1(f\,\y)\,=\,\mbox{mass($\mu_f$)}\times\,
\displaystyle\left[\begin{array}{ccc}1&\vert &{\rm E}_{\tilde{\mu}_f}(\x)^T\\ -& &-\\ {\rm E}_{\tilde{\mu}_f}(\x) &\vert & 
{\rm E}_{\tilde{\mu}_f}(\x\x^T)\end{array}\right],\end{equation}
where ${\rm E}_{\tilde{\mu}_f}(\cdot)$ denotes the expectation operator associated with the normalization $\tilde{\mu}_f$ of $\mu_f$
(as a probability measure), and ${\rm E}_{\tilde{\mu}_f}(\x\x^T)$
denotes the matrix of noncentral second-order moments of $\tilde{\mu}_f$ 
(and the covariance matrix of $\tilde{\mu}_f$ if ${\rm E}_{\tilde{\mu}_f}(\x)=0$).

\subsection{Main result}
In \cite[Theorem 3.2]{lasserre} the author has shown in a general context 
that a polynomial $f\in\R[\x]$ is nonnegative on a closed set $\K\subseteq\R^n$ if and only if
\begin{equation}
\label{aux1}
\int_\K g^2\,f\,d\mu\,\geq\,0,\qquad\forall g\in\R[\x],\end{equation}
where $\mu$ is a given finite Borel measure with support $\supmu$ being exactly $\K$; if $\K$
is compact then $\mu$ is arbitrary whereas if $\K$ is not compact then $\mu$ has to satisfy a 
certain growth condition on its moments. If $\y=(y_\alpha)$, $\alpha\in\N^n$,
is the sequence of moments of $\mu$ then  (\ref{aux1}) is in turn equivalent to 
\[\M_d(f\,\y)\,\succeq\,0,\qquad\forall d=0,1,\ldots,\]
where $\M_d(f\,\y)$ is the localizing matrix associated with $f$ and $\y$, defined in (\ref{localizing}).
In this section we particularize this result to the case of copositive matrices viewed
as homogeneous forms of degree $2$, nonnegative on the closed set $\K=\R^n_+$.\\

So with $\A=(a_{ij})\in\s^n$, let denote by $f_\A\in\R[\x]_2$ the 
quadratic form $\x\mapsto \x^T\A\x$, and let $\mu$ be the 
joint probability measure associated with $n$ {\it i.i.d.} exponential variates (with mean $1$), with support $\supmu=\R^n_+$,
and with moments $\y=(y_\alpha)$, $\alpha\in\N^n$, given by:
\begin{equation}
\label{mom}
y_\alpha=\int_{\R^n_+} \x^\alpha\,d\mu(\x)\,=\,\int_{\R^n_+} \x^\alpha\,\exp(-\sum_{i=1}^nx_i)\,d\x\,=\,\prod_{i=1}^n\alpha_i{\rm!},\qquad\forall\,\alpha\in\N^n.\\
\end{equation}
Recall that a matrix $\A\in\s^n$ is copositive if $f_\A(\x)\geq0$ for all $\x\in\R^n_+$, and denote by $\cc\subset\s^n$
the cone of copositive matrices, i.e.,
\begin{equation}
\label{defcop}
\cc\,:=\,\{\:\A\in\s^n\::\: f_\A(\x)\,\geq\,0\quad \forall\,\x\in\R^n_+\:\}.
\end{equation}
Its dual cone is the closed convex cone of
{\it completely positive}, i.e., matrices of $\s^n$ that can be written as
the sum of finitely many rank-one matrices $\x\x^T$, with $\x\in\R^n_+$, i.e.,
\begin{equation}
\label{cstar}
\cc^*\,=\,{\rm conv}\,\{\:\x\x^T\::\:\x\in\R^n_+\:\}.
\end{equation}
Next, introduce the following sets $\C_d\subset\s^n$, $d=0,1,\ldots$, defined by:
\begin{equation}
\label{def-approxd}
\C_d\,:=\,\{\:\A\in\s^n\::\:\M_d(f_\A\,\y)\,\succeq\,0\:\},\qquad d=0,1,\ldots
\end{equation}
where $\M_d(f_\A\,\y)$ is the localizing matrix defined in (\ref{localizing}), associated with the 
quadratic form $f_\A$ and the sequence $\y$ in (\ref{mom}). 

Observe that in view of the definition (\ref{localizing}) of the localizing matrix, the entries of
the matrix $\M_d(f_\A\,\y)$ are homogeneous and linear in $\A$. Therefore, 
the condition $\M_d(f_\A\,\y)\succeq0$ is a homogeneous Linear Matrix Inequality (LMI) and defines a closed 
convex cone (in fact a spectrahedron) of $\s^n$ (or equivalently, of $\R^{n(n+1)/2}$).
Each $\C_d\subset\s^n$ is a convex cone defined solely in terms of the entries $(a_{ij})$ of $\A\in\s^n$,
and the hierarchy of spectrahedra  $(\C_d)$, $d\in\N$, provides a nested sequence of outer approximations of $\cc$. 

\begin{thm}
\label{thmain}
Let $\y$ be as in (\ref{mom}) and let $\mathcal{C}_d\subset\s^n$, $d=0,1,\ldots$, be the hierarchy of convex cones defined in
(\ref{def-approxd}).
 Then $\C_0\supset\C_1\cdots\supset\C_d\cdots\supset\cc$ and $\cc=\displaystyle\bigcap_{d=0}^\infty\C_d$.
\end{thm}
The proof is a direct consequence of \cite[Theorem 3.3]{lasserre} with $\K=\R^n_+$ and $f=f_\A$.
Since $f_\A$ is homogeneous, alternatively one may use the probability measure $\nu$ uniformly supported on the $n$-dimensional simplex $\Delta=\{\x\in\R^n_+:\sum_ix_i\leq1\}$ and invoke \cite[Theorem 3.2]{lasserre}. The moments 
$\tilde{\y}=(\tilde{y}_\alpha)$ of $\nu$ are also quite simple to obtain and read:
\begin{equation}
\label{mom-nu}
\tilde{y}_\alpha\,=\,\int_{\Delta}\x^\alpha\,d\x\,=\,\frac{\alpha_1{\rm !}\cdots\alpha_n{\rm !}}{(n+\sum_i\alpha_i){\rm !}},\qquad \forall\alpha\in\N^n.\end{equation}
(See e.g. Grundmann \cite{grundmann}.)

Observe that the set membership problem ``$\A\in\cc_d$", i.e., testing whether a given matrix 
$\A\in\s^n$ belongs to $\cc_d$, is an eigenvalue problem as one has to
check whether the smallest eigenvalue of $\M_d(f_\A\,\y)$ is nonnegative. Therefore, instead of using
standard packages for Linear Matrix Inequalities, one  may use powerful specialized softwares for computing eigenvalues
of real symmetric matrices.\\

We next describe an inner approximation of the convex cone 
$\cc^*$ via the hierarchy of convex cones $(\C_d^*)$, $d\in\N$, where each $\C_d^*$ is the dual cone of $\C_d$ in Theorem \ref{thmain}. 

Recall that $\Sigma[\x]_d$ is the space of polynomials that are sums of squares of polynomials of 
degree at most $d$. A matrix $\A\in\s^n$ is also identified with a vector $\a\in\R^{n(n+1)/2}$, and conversely,
with any vector $\a\in\R^{n(n+1)/2}$ is associated a matrix $\A\in\s^n$. For instance, with $n=2$,
\begin{equation}
\label{equivalence}
\A\,=\,\left[\begin{array}{cc} a &b\\ b&c\end{array}\right]\quad\leftrightarrow\quad 
\a\,=\,\left[\begin{array}{c} a\\ 2b\\ c\end{array}\right].\end{equation}
So we will not distinguish between a convex cone in $\R^{n(n+1)/2}$ and the corresponding cone in $\s^n$.

\begin{thm}
\label{thm-dual}
Let $\C_d\subset\s^n$ be the convex cone defined in (\ref{def-approxd}). Then
\begin{equation}
\label{def1-approx-dual}
\C_d^*={\rm cl}\,\left\{\left( \la\,\X,\M_d(x_ix_j\,\y)\,\ra\,\right)_{1\leq i\leq j\leq n}\::\:\X\in\s^{s(d)}_+\right\}.
\end{equation}
Equivalently:
\begin{eqnarray}
\nonumber
\C_d^*&=&{\rm cl}\,\left\{\int_{\R^n_+} \x\x^T\,\underbrace{\sigma(\x)\,d\mu(\x)}_{d\mu_\sigma(\x)}\::\: \sigma\in\Sigma[\x]_d\:\right\}\\
\label{def2-approx-dual}
&=&{\rm cl}\,\left\{\:{\rm mass}(\mu_\sigma){\rm E}_{\tilde{\mu}_\sigma}(\x\x^T)\::\: \sigma\in\Sigma[\x]_d\:\right\},
\end{eqnarray}
with $\tilde{\mu}_\sigma$ and ${\rm E}_{\tilde{\mu}_\sigma}(\x\x^T)$ as in (\ref{second}).
\end{thm}
\begin{proof}
Let
\[\Delta_d\,:=\,\left\{\left( \la\,\X,\M_d(x_ix_j\,\y)\,\ra\,\right)_{1\leq i\leq j\leq n}\::\:\X\in\s^{s(d)}_+\right\},\]
so that
\begin{eqnarray*}
\Delta^*_d&=&\left\{\a\in\R^{n(n+1)/2}\::\:\sum_{1\leq i\leq j\leq n}a_{ij}\,\la\X,\M_d(x_ix_j\,\y\ra\geq0\quad\forall\, \X\in\s^{s(d)}_+\:\right\},\\
&=&\left\{\a\in\R^{n(n+1)/2}\::\:\left\la\X,\M_d\left((\sum_{1\leq i\leq j\leq n}a_{ij}x_ix_j)\,\y\right)\right\ra\geq0\quad\forall \,\X\in\s^{s(d)}_+\:\right\},\\
&=&\{\A\in\s^n \::\:\la\X,\M_d(f_\A\,\y)\ra\geq0\quad\forall \,\X\in\s^{s(d)}_+\:\}\quad\mbox{[with $\A,\a$ as in (\ref{equivalence})]}\\
&=&\{\A\in\s^n \::\:\M_d(f_\A\,\y)\,\succeq0\:\}\,=\,\C_d.\end{eqnarray*}
And so we obtain the desired result $\C_d^*=(\Delta_d^*)^*={\rm cl}\,(\Delta_d)$.
Next, writing the singular decomposition of $\X$ as $\sum_{k=0}^s\q_k\,\q_k^T$ for some $s\in\N$ and some
vectors $(\q_k)\subset\R^{s(d)}$,
one obtains that for every $1\leq i\leq j\leq n$,
\begin{eqnarray*}
\la\,\X,\M_d(x_ix_j\,\y)\,\ra &=&\sum_{k=0}^s\la\,\q_k\q_k^T,\M_d(x_ix_j\,\y)\,\ra\,=\,
\sum_{k=0}^s\la\,\q_k,\M_d(x_ix_j\,\y)\q_k\,\ra\\
&=&\sum_{k=0}^s\int_{\R^n_+} x_ix_j \,q_k(\x)^2\,d\mu(\x)\quad\mbox{[by (\ref{local2})]}\\
&=&\int_{\R^n_+} x_ix_j\,\underbrace{\sigma(\x)\,d\mu(\x)}_{d\mu_\sigma(\x)},\end{eqnarray*}
where $\sigma(\x)=\sum_{k=0}^sq_k(\x)^2\in\Sigma[\x]_d$, and $\mu_\sigma(B)=\int_B\sigma\,d\mu$ for all Borel sets $B$.
\end{proof}
So Theorem \ref{thm-dual} states that $\C_d^*$ is the closure of the convex cone generated by
second-order moments of measures $d\mu_\sigma=\sigma d\mu$, absolutely continuous with respect to $\mu$ (hence with support on $\R^n_+$) and with density being a s.o.s. polynomial $\sigma$ of degree at most $2d$.
Of course we immediately have:
\begin{cor}
\label{cor2}
Let $\C^*_d$, $d\in\N$, be as in (\ref{def2-approx-dual}). Then $\C_d^*\subset\C^*_{d+1}$ for all $d\in\N$, and
\[\cc^*\,=\,{\rm cl}\,\bigcup_{d=0}^\infty\C^*_d.\]
\end{cor}
\begin{proof}
As $\C^*_d\subset\C^*_{d+1}$ for all $d\in\N$, the result follows from
\[\cc^*\,=\,\left(\,\bigcap_{d=0}^\infty\,\C_d\,\right)^*\,=\,{\rm cl}\,\left({\rm conv}\,\bigcup_{d=0}^\infty\,\C^*_d\,\right)\,=\,
{\rm cl}\,\bigcup_{d=0}^\infty\,\C^*_d.\]
\end{proof}
In other words, $\C_d^*$ approximates $\cc^*={\rm conv}\,\{\x\x^T:\x\in\R^n_+\}$ (i.e., the convex hull of second-order moments of Dirac measures
with support in $\R^n_+$) from inside by second-order moments of measures $\mu_\sigma\ll\mu$ whose density is a s.o.s.
polynomial $\sigma$ of degree at most $2d$, and better and better approximations are obtained by letting $d$ increase.
\begin{ex}
{\rm  For instance, with $n=2$, and 
$\A=\left[\begin{array}{cc}a & b\\b&c\end{array}\right]$, 
it is known that $\A$ is copositive if and only if $a,c\geq0$ and $b+\sqrt{ac}\geq0$; see e.g. \cite{jbhu-sirev}.
Let $\mu$ be the exponential measure on $\R^2_+$ with moments defined in (\ref{mom}).
With $f_\A(\x):=ax_1^2+2bx_1x_2+cx_2^2$ and $d=1$,
the condition $\M_d(f_\A\,\y)\succeq0$ which reads
\begin{equation}
\label{mat1}
2\left[\begin{array}{ccc}
a+b+c & 3a+2b+c& a+2b+3c\\ 
3a+2b+c & 12a+6b+2c &3a+4b+3c\\
a+2b+3c&3a+4b+3c&2a+6b+12c\end{array}\right]\,\succeq\,0,\end{equation}
defines the convex cone $\C_1\subset\R^3$. It is a connected component of the basic semi-algebraic set
$\{(a,b,c): {\rm det}\,(\M_1(f_\A\,\y))\geq0\}$, that is, elements $(a,b,c)$ such that:
\begin{equation}
\label{det}
3a^3+15a^2b+29a^2c+16ab^2+50abc+29ac^2+4b^3+16b^2c+15bc^2+3c^3\,\geq\,0.\end{equation}
Alternatively,  by homogeneity, instead of $\mu$ we may 
take the probability measure $\nu$ uniformly supported on the simplex $\Delta$
and with moments $\tilde{\y}=(\tilde{y}_\alpha)$ given in (\ref{mom-nu}), in which case 
the corresponding matrix $\M_1(f_\A\,\tilde{\y})$ now reads:
\begin{equation}
\label{mat2}
\frac{1}{360}\left[\begin{array}{ccc}
30(a+b+c) & 6(3a+2b+c)& 6(a+2b+3c)\\ 
6(3a+2b+c) & 12a+6b+2c &3a+4b+3c\\
6(a+2b+3c)&3a+4b+3c&2a+6b+12c\end{array}\right]\,\succeq\,0.\end{equation}
It turns out that up to a multiplicative factor both matrices (\ref{mat1}) and (\ref{mat2}) have same determinant and so define the same 
convex cone $\C_1$ (the same connected component of (\ref{det})).
Figure 1 below displays the projection on the $(a,b)$-plane of the sets $\C_1$ and $\cc$ intersected with the unit ball.
\begin{center} 
\begin{figure}
\resizebox{0.5\textwidth}{!}
%\resizebox{\textwidth}{!}
{\includegraphics{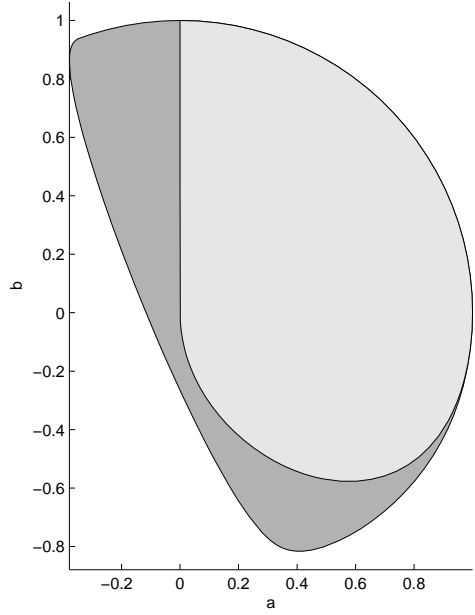}}
\caption{$n=2$: Projection on the $(a,b)$-plane of $\cc$ versus $\C_1$, both intersected with the unit ball}%\label{fig1}
\end{figure}
\end{center}
}\end{ex}

\subsection{An alternative representation of the cone $\C^*_d$}

From its definition (\ref{def1-approx-dual}) 
%(\ref{def2-approx-dual})
in Theorem \ref{thm-dual}, the cone $\C^*_d\subset\s^n$ is defined through the
matrix variable $\X\in\s^{s(d)}$ which lives in a (lifted) space of dimension $s(d)(s(d)+1)/2$ 
and with the linear matrix inequality (LMI)  constraint $\X\succeq0$ of size $s(d)$. In contrast,
the convex cone $\C_d\subset\s^n$ defined in (\ref{def-approxd}) is defined solely in terms of the entries of 
the matrix $\A\in\s^n$, that is, no projection from a higher dimensional space (or no lifting) is needed.

We next provide another explicit description on how $\C^*_d$ can be generated with no LMI constraint
and with only $s(d)$ variables, but of course this characterization is not well suited for optimization purposes.

Since every $g\in\Sigma[\x]_d$ can be written $\sum_\ell g_\ell^2$ for finitely many 
polynomials $(g_\ell)\subset\R[\x]_d$, the convex cone $\Sigma[\x]_d$ of s.o.s. polynomials can be written
\[\Sigma[\x]_d\,=\,{\rm conv}\,\{g^2\::\:g\in\R[\x]_d\:\}.\]
Next, for $g\in\R[\x]_d$, with vector of coefficients $\g=(g_\alpha)\in\R^{s(d)}$, let
$\g^{(2)}=(g^{(2)}_\alpha)\in\R^{s(2d)}$ be the vector of coefficients of $g^2$, that is,
\[g(\x)=\sum_{\alpha\in\N^n_d}g_\alpha\,\x^\alpha \quad\rightarrow\quad g(\x)^2\,=\,\sum_{\alpha\in\N^n_{2d}}g^{(2)}_\alpha\,\x^\alpha.\]
Notice that for each $\alpha\in\N^n_{2d}$, $g^{(2)}_\alpha$ is quadratic in $\g$.
For instance, with $n=2$ and $d=1$, $g(\x)=g_{00}+g_{10}x_{1}+g_{01}x_{2}$ with $\g=(g_{00},\,g_{10},g_{01})^T\in\R^{s(1)}$, and so
\[\g^{(2)}\,=\,(g_{00}^2,\, 2g_{00}g_{10},\,2g_{00}g_{01},\,g_{10}^2,\,2g_{10}g_{01},\,g_{02}^2)^T\in\R^{s(2)}.\]

Next, for $g\in\R[\x]_d$ and every $1\leq i\leq j\leq n$, let $\G_d=(G_{ij})\in\s^n$ be defined by:
\begin{equation}
\label{definal}
G_{ij}\,:=\,\int_{\R^n}g(\x)^2\,x_ix_j\,d\mu(\x)\,=\,\sum_{\alpha\in\N^n_{2d}}g^{(2)}_\alpha\,(\alpha_i+1){\rm !}(\alpha_j+1){\rm !}\prod_{k\neq i,j}\alpha_k{\rm !},
\end{equation}
for all $1\leq i\leq j\leq n$. We can now describe $\C^*_d$.
\begin{cor}
\label{cor3}
Let $\G_d\in\s^n$ be as in (\ref{definal}). Then:
\begin{equation}
\label{cor3-1}
\C^*_d\,=\,{\rm cl}\,\left({\rm conv}\,\{\G_d\::\: \g\in\R^{s(d)}\:\}\,\right).\end{equation}
\end{cor}
\begin{proof}
From Theorem \ref{thm-dual} 
\[\C_d^*={\rm cl}\,\left\{\left( \la\,\X,\M_d(x_ix_j\,\y)\,\ra\,\right)_{1\leq i\leq j\leq n}\::\:\X\in\s^{s(d)}_+\right\}.\]
As in the proof of Theorem \ref{thm-dual}, writing $\X\succeq0$ as $\sum_{k=1}^s\g_k\g_k^T$ for some vectors 
$(\g_k)\subset\R^{s(d)}$ (and associated polynomials $(g_k)\subset\R[\x]_d$),
\[\la\,\X,\M_d(x_ix_j\,\y)\,\ra\,=\,\sum_{k=1}^s \int_{\R^n_+} x_ix_jg_k(\x)^2\,d\mu,
%&=&\sum_{k=1}^s \sum_{\alpha\in\N^n_{2d}}g^{(2)}_{k\alpha}\,(\alpha_i+1){\rm !}(\alpha_j+1){\rm !}\prod_{k\neq i,j}\alpha_k{\rm !}
\,=\,\sum_{k=1}^s G^k_{ij},\]
for all $1\leq i\leq j\leq n$, where $G^k_{ij}$ is as in (\ref{definal}) (but now associated with $g_k$ instead of $g$). Hence
\[\left(\la\,\X,\M_d(x_ix_j\,\y)\,\ra\,\right)_{1\leq i\leq j\leq n}=\sum_{k=1}^s \G^k_d\,\in\,{\rm conv}\,\{\G_d\::\: \g\in\R^{s(d)}\:\}.\]
Hence $\cc^*\subseteq {\rm cl}\,\left({\rm conv}\,\{\G_d\::\: \g\in\R^{s(d)}\:\}\right)$.

Conversely, let $h\in {\rm conv}\,\{\G_d\::\: \g\in\R^{s(d)}\:\}$, i.e., for some integer $s$ and polynomials $(g_k)\subset\R[\x]_d$,
and for all $1\leq i\leq j\leq n$, 
\begin{eqnarray*}
h_{ij}\,=\,\sum_{k=1}^s \underbrace{\lambda_k}_{>0}\,G^k_{ij}&=&\int_{\R^n_+}\left(\sum_{k=1}^s\lambda_k\,g_k^2(\x)\right)\,x_ix_j\,d\mu(\x)\\
&=&\la \X,\M_d(x_ix_j\,\y)\ra 
\end{eqnarray*}
where $\X=\sum_{k=1}^s\lambda_k\,\g_k\g_k^T\in\s^{s(d)}_+$. Hence 
${\rm cl}\,\left({\rm conv}\,\{\G_d\::\: \g\in\R^{s(d)}\:\}\right)\subseteq\cc^*$.
\end{proof}
The characterization (\ref{cor3-1}) of $\C^*_d$ should be compared with 
the characterization ${\rm conv}\,\{\x\x^T\,:\:\x\in\R^n_+\}$ of $\cc^*$.
\begin{ex}
With $n=2$ and $d=1$, $\G_1$ reads:
{\scriptsize
\[\left[\begin{array}{cc}
2g^2_{00}+12g_{10}(g_{00}+g_{01})+4g_{01}(g_{00}+g_{01})+24g_{10}^2&
g_{00}^2+4g_{00}(g_{10}+g_{01})+6(g_{10}^2+g_{01}^2)+8g_{10}g_{01}\\
g_{00}^2+4g_{00}(g_{10}+g_{01})+6(g_{10}^2+g_{01}^2)+8g_{10}g_{01}&
2g^2_{00} +4g_{10}(g_{00}+g_{10})+12g_{01}(g_{00}+g_{10})+24g_{01}^2\end{array}\right]\]
}
\end{ex}

\subsection{$\k$-copositive and $\k$-completely positive matrices}

Let $\k\subset\R^n$ be a closed convex cone and let 
$\cc_\k$ be the convex cone of $\k$-copositive matrices, i.e., matrices $\A\in\s^n$ such that $\x^T\A\x\geq0$ on $\k$.
Its dual cone $\cc^*_\k\subset\s^n$ is the cone of $\k$-completely positive matrices.

Then one may define a hierarchy of convex cones $(\C_{\k})_d$ and $(\C^*_{\k})_d$, $d\in\N$, formally exactly as in Theorem \ref{thmain}
and Theorem \ref{thm-dual}, but now $\y$ is the moment sequence of a finite Borel measure 
$\mu$ with $\supmu=\k$ (instead of $\supmu=\R^n_+$ in (\ref{mom})), i.e., $\y=(y_\alpha)$, $\alpha\in\N^n$, with:
\begin{equation}
\label{newmom}
y_\alpha\,:=\,\int_\k\,\x^\alpha\,d\mu,\qquad \forall\alpha\in\N^n\quad\mbox{ (and where $\supmu=\k$)}.
\end{equation}

And so with $\y$ as in (\ref{newmom}) Theorem \ref{thmain} and \ref{thm-dual}, as well as Corollary \ref{cor2}, are still valid.
(In (\ref{def2-approx-dual} replace $\int_{\R^n_+}$ with $\int_\k$). Therefore,
\[\cc_\k\,=\,\bigcap_{d=0}^\infty (C_{\k})_d\quad\mbox{ and }\quad \cc^*_\k\,=\,{\rm cl}\,\bigcup_{d=0}^\infty(\C_{\k})_d^*.\]
But of course, for practical implementation, one need to know the sequence $\y=(y_\alpha)$, $\alpha\in\N^n$, which was easy when $\k=\R^n_+$.
For instance, if $\k$ is a polyhedral cone, by homogeneity of $f_\A$, one may equivalently 
consider a compact base of $\k$ (which is a polytope $\k'$), and take for $\mu$ the Lebesgue measure on $\k'$.
Then all  moments of $\mu$ can be calculated exactly. The same argument works for every
convex cone $\k$ for which one may compute all moments of a finite Borel measure $\mu$
whose support is a compact base of $\k$.

\section*{Acknowledgement}
The author wishes to thank two anonymous referees for their very helpful remarks and suggestions 
to improve the initial version of this note.

\end{document}